\documentclass[12pt]{amsart}
\usepackage[margin=1in]{geometry}
\usepackage{hyperref}
\usepackage{amsfonts}
\usepackage{amssymb, mathrsfs}
\usepackage{amsopn}
\usepackage{amsmath,amscd}

\title{On fully residually-$\mathcal{R}$ groups}
\author{Inna Bumagin and Ming Ming Zhang}
\address{School of Mathematics and Statisics\\
               Carleton University\\
               Ottawa, ON, Canada  K1S 5B6\\
               E-mail addresses : \href{mailto:bumagin@math.carleton.ca}{bumagin@math.carleton.ca}
\href{mailto:mingzhang@cmail.carleton.ca}{mingzhang@cmail.carleton.ca}}

\newtheorem{theorem}{Theorem}[section]
\newtheorem{definition}[theorem]{Definition}
\newtheorem{lemma}[theorem]{Lemma}
\newtheorem{corollary}[theorem]{Corollary}
\newtheorem{proposition}[theorem]{Proposition}
\newtheorem{remark}[theorem]{Remark}
\newtheorem{example}[theorem]{Example}
\newtheorem*{question}{Question}

\begin{document}

\begin{abstract} We consider the class $\mathcal{R}$ of finitely generated toral relatively hyperbolic groups. We show that groups from $\mathcal{R}$ are commutative transitive and generalize a theorem proved by Benjamin Baumslag in~\cite{Baumslag65} to this class. We also discuss two definitions of (fully) residually-$\mathcal{C}$ groups, the classical Definition \ref{def:resCclassical} and a modified Definition \ref{def:resCworking}. Building upon results obtained by Ol'shanskii~\cite{Olshanskii93} and Osin~\cite{Osin2010}, we prove the equivalence of the two definitions for $\mathcal{C}=\mathcal{R}$. This is a generalization of the similar result obtained by Ol'shanskii for $\mathcal{C}$ being the class of torsion-free hyperbolic groups. Let $\Gamma\in\mathcal{R}$ be non-abelian and non-elementary. Kharlampovich and Miasnikov proved in ~\cite{KharlampovichMyasnikov09} that a finitely generated fully residually-$\Gamma$ group $G$ embeds into an iterated extension of centralizers of $\Gamma$. We deduce from their theorem that every finitely generated fully residually-$\Gamma$ group embeds into a group from $\mathcal{R}$. On the other hand, we give an example of a finitely generated torsion-free fully residually-$\mathcal{H}$ group that does not embed into a group from $\mathcal{R}$; $\mathcal{H}$ is the class of hyperbolic groups. 
\end{abstract}

\maketitle


\section{Introduction}

The notion of a (fully) residually-$\mathcal{C}$ group, where $\mathcal{C}$ is a class of groups, was introduced long time ago. Usually, $\mathcal{C}$ is chosen to be a class of groups with nice properties, such as the class of all finite groups, nilpotent groups, free groups, etc.  The classical definitions are as follows.

\begin{definition}\label{def:resCclassical}\rm \textbf{(Classical Definition.)}
Let $\mathcal{C}$ be a class of groups. A group $G$ is called a \textit{residually}-$\mathcal{C}$ \textit{group} if for every nontrivial element $1\neq g\in G$ there is a group $H_g\in\mathcal{C}$ and an onto homomorphism $\phi_g\colon G\rightarrow H_g$, such that $\phi_g(g)\neq 1$. A group $G$ is called a \textit{fully residually-$\mathcal{C}$ group} if for every finite set $S=\{g_1,\dots,g_n\}$ of distinct elements of $G$ there is a group $H_S\in\mathcal{C}$ and an onto homomorphism $\phi_S\colon G\rightarrow H_S$, such that the images $\phi_S(g_1),\dots,\phi_S(g_n)$ are all distinct in $H_S$.
\end{definition}

Clearly, every fully residually-$\mathcal{C}$ group is, in particular, residually-$\mathcal{C}$. For some classes $\mathcal{C}$, the opposite also holds. It is not hard to show that every residually finite group  is fully residually finite and that every residually nilpotent group is fully residually nilpotent. The situation is very different for the class of free groups, as was shown by Benjamin Baumslag.

\begin{theorem} (B. Baumslag~\cite{Baumslag65}) \label{thm:BB}
 A group is fully residually free if and only if it is residually free and commutative transitive.
\end{theorem} 

Recall that a group $G$ is called \textit{commutative transitive} if for any nontrivial elements $g,h,f\in G$, if $[g,h]=1$ and $[g,f]=1$ then $[h,f]=1$.

We generalize Baumslag's theorem to the wider class $\mathcal{R}$ of finitely generated toral relatively hyperbolic groups defined in Section~\ref{sec:RelHypGps}. In Section~\ref{sec:BaumslagThm} we prove the following theorem.

\begin{theorem} \label{thm:generalizedBB}
A finitely generated group is fully residually-$\mathcal{R}$ if and only if it is residually-$\mathcal{R}$ and commutative transitive. 
\end{theorem}

A generalization of Theorem~\ref{thm:BB} to some classes of groups was obtained by Ciobanu, Fine and Rosenberger in~\cite{CFR}. Whereas some of the groups they consider are in the class $\mathcal{R}$, other groups have torsion and thus are not covered by our Theorem~\ref{thm:generalizedBB}.

Our motivation is rooted in the growing interest to the algebraic geometry over toral relatively hyperbolic groups. 
Fully residually free groups play a critical role in the theory of equations over free groups and therefore, in the solution to the famous Tarski's problems~\cite{SelaElemTheoryFreeGps06},~\cite{KharlampovichMyasnikov06}. We refer the interested reader to~\cite{KharlampovichMyasnikov98p1},~\cite{KharlampovichMyasnikov98},~\cite{BaumsalgMyasnikovRemeslennikov99}, where the basics of the algebraic geometry over free groups are developed. Later, the solution to the Tarski problem was generalized to torsion-free hyperbolic groups~\cite{SelaElemTheoryHypGps},~\cite{KharlampovichMyasnikov13}. Sela introduced limit groups and showed that limit groups are precisely finitely generated fully residually free groups~\cite{SelaMRDiagrams01}. 

However, one changes the point of view slightly when studying algebraic geometry over (relatively) hyperbolic groups.
Firstly, one has to consider all homomorphisms, not only epimorphisms; in other words, one uses the following  Definition~\ref{def:resCworking}, rather than Definition~\ref{def:resCclassical}. 

\begin{definition}\label{def:resCworking}\rm \textbf{(Working Definition.)}
Let $\mathcal{C}$ be a class of groups. A group $G$ is called a \textit{residually}-$\mathcal{C}$ \textit{group} if for every nontrivial element $1\neq g\in G$ there is a group $H_g\in\mathcal{C}$ and a homomorphism $\phi_g\colon G\rightarrow H_g$, such that $\phi_g(g)\neq 1$. A group $G$ is called a \textit{fully residually-$\mathcal{C}$ group} if for every finite set $S=\{g_1,\dots,g_n\}$ of distinct elements of $G$ there is a group $H_S\in\mathcal{C}$ and a homomorphism $\phi_S\colon G\rightarrow H_S$, such that the images $\phi_S(g_1),\dots,\phi_S(g_n)$ are all distinct in $H_S$.
\end{definition} 

The difference with the Classical Definition~\ref{def:resCclassical} is that the homomorphisms $\phi_g$ and $\phi_S$ are no longer required to be onto. Clearly, the classes of (fully) residually finite and (fully) residually free groups are the same for both definitions. The equivalence of the definitions for $\mathcal{C}$ being the class of torsion-free hyperbolic groups and $G$ being finitely generated is immediate from results of Ol'shanskii~\cite{Olshanskii93}. In Section~ \ref{sec:equivalent-def} we show that the definitions are equivalent for the class of finitely generated toral relatively hyperbolic groups. Our proof is based on statements that generalize results of Ol'shanskii and are proved by Osin~\cite{Osin2010}. The equivalence of the definitions is an immediate corollary of the following theorem.

\begin{theorem} Let $\mathcal{R}$ denote the class of finitely generated toral relatively hyperbolic groups, and let $\Gamma$ be a group in $\mathcal{R}$. Then every finitely generated subgroup of $\Gamma$ is a fully residually-$\mathcal{R}$ group in the sense of Definition~\ref{def:resCclassical}. 
\end{theorem}

The second change that one makes to extend the algebraic geometry over free groups to (relatively) hyperbolic groups, is that one's attention restricts to homomorphisms into a fixed group $\Gamma$ from the given class. Accordingly, the terminology changes to (fully) residually-$\Gamma$ groups and $\Gamma$-limit groups. 

In this context, $\Gamma$-limit groups are precisely finitely generated fully residually-$\Gamma$ groups, when the group $\Gamma$ is either hyperbolic, as was shown by Sela in~\cite{SelaElemTheoryHypGps} for the torsion-free case and by Reinfeldt and Weidmann in~\cite{RW} for all hyperbolic groups, or $\Gamma\in\mathcal{R}$, as was proved by Groves in~\cite{GrovesLimitRelHypGps05}. Yet more generally, the same result holds if $\Gamma$ is an equationally Notherian group, as was proved by Ould Houcine in~\cite{OuldH}.

We do not need to restrict to a fixed group $\Gamma\in\mathcal{R}$ to prove the generalization of Baumslag's theorem. However, we do that when we try to study the class of fully residually-$\mathcal{R}$ groups.

Note that usually, finitely generated (fully) residually-$\mathcal{C}$ groups do not have to be in $\mathcal{C}$. For instance, infinite linear groups are residually finite, and finitely generated free abelian groups are fully residually free. However, for some classes $\mathcal{C}$, a group $G$ is fully residually-$\mathcal{C}$ if and only if $G\in\mathcal{C}$. An example is the class $\mathcal{A}$ of all abelian groups. Indeed, let a group $G$ be fully residually-$\mathcal{A}$, and suppose there are $a,b\in G$ such that $[a,b]\neq 1$. Clearly, both $a$ and $b$ are nontrivial. Then there is an abelian group $H$ and a homomorphism $\phi\colon G\rightarrow H$ such that $\phi(a)$, $\phi(b)$ and $\phi([a,b])=[\phi(a),\phi(b)]$ are all nontrivial in $H$, which is a contradiction. It follows that $G$ is fully residually-$\mathcal{A}$ if and only if $G\in\mathcal{A}$. In Section \ref{sec:Toral-LimitGps} we show that the class of finitely generated toral relatively hyperbolic groups is closed under extensions of centralizers (see Definition~\ref{def:ext-centralizers}) and deduce the following theorem from a statement proved by Kharlampovich and Miasnikov in~\cite{KharlampovichMyasnikov09}.

\begin{theorem}\label{thm:fresGammaembeds}
Let $\Gamma$ be a finitely generated toral relatively hyperbolic group, and let $G$ be a finitely generated fully residually-$\Gamma$ group. Then $G$ embeds into a toral relatively hyperbolic group.
\end{theorem}

Note that finitely generated fully residually free groups are toral relatively hyperbolic; this follows from combination theorems proved by Dahmani~\cite{DahmaniCombinationThm} and Alibegovi\v{c}~\cite{Alibegovic}. On the other hand, neither this latter statement, nor even the conclusion of Theorem~\ref{thm:fresGammaembeds} holds for torsion-free finitely generated fully residually-$\mathcal{H}$ groups, where $\mathcal{H}$ is the class of hyperbolic groups. Indeed, there are Baumslag-Solitar fully residually-$\mathcal{H}$ groups that do not embed into any group in $\mathcal{R}$; please see Section \ref{sec:BS-gps} for details.

\section{Relatively hyperbolic groups}\label{sec:RelHypGps}

\subsection{Definitions}

We begin with the definition of a relatively hyperbolic group that was originally introduced by Osin in \cite{OsinBook}. 

Let $G$ be a group, $\{H_{\lambda}, \lambda \in \Lambda\}$ be a collection of subgroups of $G$, and $X$ be a subset of $G$. We say that $X$ is a \emph{relative generating set of $G$ with respect to $\{H_{\lambda}, \lambda \in \Lambda\}$} if $G$ is generated by the set $(\bigcup_{\lambda \in \Lambda}H_{\lambda}) \cup X$ where $X$ is assumed to be symmetric, i.e., $X = X^{-1}$. In this situation, there exists a natural homomorphism $$\eta  : F_{G} = (*_{\lambda \in \Lambda}H_{\lambda})*F(X) \rightarrow G,$$ where $F(X)$ is the free group with the basis $X$. Let $N$ denote the kernel of $\eta$. If $N$ is a normal closure of a subset $\mathscr{R} \subseteq N$ in the group $F_{G}$, we say that $G$ has the \emph{relative presentation} $$\langle X, H_{\lambda}, \lambda \in \Lambda \mid R = 1, R \in \mathscr{R}\rangle $$ \emph{with respect to} $\{H_{\lambda}, \lambda \in \Lambda\}$. For brevity, we write $$\langle X, H_{\lambda}, \lambda \in \Lambda \mid \mathscr{R}\rangle .$$ The relative presentation is \emph{finite} if the sets $X$ and $\mathscr{R}$ are finite. The group $G$ is called \emph{finitely presented relative to $\{H_{\lambda}, \lambda \in \Lambda\}$} if $G$ has a finite relative presentation  with respect to $\{H_{\lambda}, \lambda \in \Lambda\}$.

Let $\mathscr{H} = \bigsqcup_{\lambda \in \Lambda}(H_{\lambda} \setminus \{1\})$. Given a word $W$ over $X \cup \mathscr{H}$ such that  $W$ represents the identity in $G$, there exists an expression $$W =_{F_{G}} \prod_{i=1}^{k}f_{i}^{-1}R_{i}^{\pm 1}f_{i}$$ with the equality in the group $F_{G}$, where $R_{i} \in \mathscr{R}$, $f_{i} \in F_{G}$ for $i = 1, \cdots , k$. The smallest possible $k$ is called the \emph{relative area} of $W$ and is denoted by $Area^{rel}(W)$.

\begin{definition}\rm
We say that $G$ is \emph{hyperbolic relative to $\{H_{\lambda}, \lambda \in \Lambda\}$} if $G$ is finitely presented relative to $\{H_{\lambda}, \lambda \in \Lambda\}$ and there is a constant $\alpha > 0$ such that for any word $W$ of finite length over $X \cup \mathscr{H}$, representing the identity in $G$, we have $Area^{rel}(W) \leq \alpha\lVert W \rVert_{X \cup \mathscr{H}}$, where $\lVert W \rVert_{X \cup \mathscr{H}}$ denotes the shortest length of the word from the free monoid generated by $X \cup \mathscr{H}$ and is called the \emph{relative length of} $W$ \emph{with respect to} $\mathscr{H}$.
\end{definition}

\begin{remark}\label{remark:fg-hyp}
Note that the collection $\{H_{\lambda}, \lambda \in \Lambda\}$ does not have to be finite and that $G$ and $H_{\lambda}$ for all $\lambda \in \Lambda$ need not be finitely presented or even finitely generated. However, if the group $G$ is finitely generated and finitely presented relative to $\{H_{\lambda}, \lambda \in \Lambda\}$, then by \cite[Theorem 1.1]{OsinBook}, the collection of subgroups is finite and each subgroup $H_{\lambda}$ is finitely generated. In this case, Osin's definition is equivalent to the definition of Bowditch \cite{Bowditch98} and to that of Farb \cite{Farb98} (including the BCP property), by \cite[Appendix]{OsinBook}. Dahmani and Guirardel proved in~\cite{DahmaniGuirardel} that if $G$ is finitely presented then each subgroup $H_{\lambda}$ is finitely presented.
\end{remark}

\begin{definition}\rm \label{def:parabolic}
Let a group $G$ be hyperbolic relative to a collection of subgroups $\mathcal{H}=\{H_{\lambda}, \lambda \in \Lambda\}$. By a \emph{maximal parabolic subgroup} of $G$ we mean any conjugate $P_{g,\lambda} = gH_{\lambda}g^{-1}$ for $g \in G$ and $\lambda\in \Lambda$. A subgroup is \textit{parabolic} if it is conjugate into one of the subgroups in $\mathcal{H}$. An element $x$ of $G$ is called \emph{parabolic} if it lies in some parabolic subgroup of $G$. Otherwise, $x$ is called \emph{hyperbolic}. 
\end{definition}

\begin{definition} \label{def:toral}\rm
A torsion-free group hyperbolic relative to a collection of abelian subgroups is called a \emph{toral relatively hyperbolic group}. We denote by $\mathcal{R}$ the class of finitely generated toral relatively hyperbolic groups.
\end{definition}

\subsection{Some properties}

Throughout this section, we fix a group $G$ hyperbolic relative to a collection of infinite subgroups $\mathcal{H}=\{H_{\lambda}, \lambda \in \Lambda\}$.

Recall that a subgroup $H$ of a group $\Gamma$ is called \emph{malnormal} if $H \cap gHg^{-1}$ is trivial for any $g \in \Gamma \setminus H$ and that $H$ is \emph{almost malnormal} if $H \cap gHg^{-1}$ has finitely many elements for any $g \in \Gamma \setminus H$. A group is called \emph{elementary} (or \emph{virtually cyclic}) if it contains a cyclic subgroup of finite index. Theorem~\ref{thm:Osin} below is a combination of several statements proved by Osin: \cite[Proposition 2.36]{OsinBook} and \cite[Theorem~1.5, Theorem~4.3, Corollary 1.6 and Corollary 1.7]{Osin04}.

\begin{theorem}\label{thm:Osin} (Osin) Let $G$ be a group hyperbolic relative to a collection $\mathcal{H}=\{H_{\lambda}, \lambda \in \Lambda\}$ of subgroups.
\begin{itemize}
\item[(1)] For any $g_{1}, g_{2} \in G$, the intersection $g_{1}H_{\lambda}g^{-1}_{1} \cap g_{2}H_{\mu}g^{-1}_{2}$ is finite whenever $\lambda,\mu\in\Lambda$ and $\lambda \neq \mu$. The intersection $gH_{\lambda}g^{-1} \cap H_{\lambda}$ is finite for any $g \notin H_{\lambda}$; in particular, every subgroup from $\mathcal{H}$ is almost malnormal.
\item[(2)] Every hyperbolic element $g$ of infinite order in $G$ is contained in a unique maximal elementary subgroup, namely in $$E(g) = \{f \in G; f^{-1}g^{n}f = g^{\pm n} \text{ for some positive integer } n\}.$$ The subgroup $E(g)$ is hyperbolic and almost malnormal.
\item[(3)] The group $G$ is hyperbolic relative to $\mathcal{H} \cup \{E(g)\}$.
\end{itemize}
\end{theorem}

We have immediate consequences as follows.

\begin{corollary}\label{cor:hyp-rel-par}
The collection $\mathcal{H}$ contains exactly one representative of each conjugacy class of maximal infinite parabolic subgroups of $G$.
\end{corollary}

\begin{proof}
It is immediate from Theorem~\ref{thm:Osin} (1) that infinite subgroups $H_{\lambda}\in\mathcal{H}$ and $H_{\mu}\in\mathcal{H}$ belong to distinct conjugacy classes, for all $\lambda\neq\mu$. On the other hand, every maximal parabolic subgroup $P$ of $G$ belongs to the conjugacy class of one of the subgroups in $\mathcal{H}$, by Definition~\ref{def:parabolic}. So, each conjugacy class of maximal parabolic subgroups has to have a representative in $\mathcal{H}$.
\end{proof}

\begin{corollary}\label{cor:parabolic-malnormal}
Every maximal parabolic subgroup $P$ of $G$ is almost malnormal. If $G$ is torsion-free then $P$ is malnormal.
\end{corollary}

\begin{proof}
Let $P = gH_{\lambda}g^{-1}$ for some $g \in G$ and $H_{\lambda} \in \mathcal{H}$; let $x \in G \setminus P$. Then $g^{-1}xg \notin H_{\lambda}$. Thus, $g^{-1}(P \cap xPx^{-1})g = H_{\lambda} \cap (g^{-1}xg)H_{\lambda}(g^{-1}xg)^{-1}$ has finitely many elements because $H_{\lambda}$ is almost malnormal, by Theorem \ref{thm:Osin} (1). Hence, $P$ is almost malnormal. The second part of the  statement is straightforward.
\end{proof}

\begin{corollary}\label{cor:parabolic-element}
Let $x$ be a parabolic element of infinite order in $G$. Suppose that $[x,y] = 1$ for $y \in G$. Then $y$ lies in the same maximal parabolic subgroup as $x$. 
\end{corollary}

\begin{proof}
Let $P$ be the maximal parabolic subgroup containing $x$. Assume that $y \notin P$. Note that $x \in yPy^{-1}$ as $[x,y]=1$. Then $x \in P \cap yPy^{-1}$, so that the intersection has infinitely many elements, which is a contradiction.  Therefore, $y \in P$.
\end{proof}

We denote the normalizer of a subgroup $H$ of a group $\Gamma$ by $N_{\Gamma}(H)$.

\begin{corollary}\label{cor:normalizer}
If $P$ is an infinite maximal parabolic subgroup of $G$, then $N_{G}(P) = P$. If $g$ is an infinite order hyperbolic element of $G$ then $N_{G}(E(g)) = E(g)$.
\end{corollary}

\begin{proof}
Suppose that $x \in N_{G}(P) \setminus P$. Then the intersection $P \cap xPx^{-1} = P$ is infinite, while $P$ is almost malnormal, by Corollary \ref{cor:parabolic-malnormal}. The contradiction implies that $N_{G}(P) \subseteq P$. Clearly, $P \subseteq N_{G}(P)$. So, $N_{G}(P) = P$. The second assertion now follows because $G$ is hyperbolic relative to $\{H_{\lambda}, \lambda \in \Lambda\} \cup \{E(g)\}$, by Theorem~\ref{thm:Osin}(3).
\end{proof}

\begin{corollary}
Let $g$ be a hyperbolic element of infinite order in $G$. Suppose that $A$ is an infinite subgroup of $E(g)$. Then $N_{G}(A)$ is a subgroup of $N_{G}(E(g))$. 
\end{corollary}

\begin{proof}
Assume that $x \in N_{G}(A) \setminus N_{G}(E(g))$. By Corollary \ref{cor:normalizer}, $N_{G}(E(g)) = E(g)$; and so $x \notin E(g)$. Thus, the cardinality of $A = xAx^{-1} \cap A \subseteq xE(g)x^{-1} \cap E(g)$ is infinite, which contradicts to $E(g)$ being almost malnormal by Theorem \ref{thm:Osin} (2). Therefore, $N_{G}(A) \subseteq N_{G}(E(g))$.
\end{proof}

\begin{lemma}\label{lemma:elementary-hyp}
Suppose that $G$ is torsion-free and that $g$ is a hyperbolic element of $G$. Then the maximal elementary subgroup $E(g) = \langle a \rangle$ is the maximal cyclic subgroup containing $g$. In particular, the centralizer $C_{G}(g) = \langle a \rangle$ of $g$ is cyclic.
\end{lemma}

\begin{proof} 
By Theorem \ref{thm:Osin} (2), $E(g)$ is a virtually cyclic group. It is well-known that every torsion-free virtually cyclic group is either trivial or infinite cyclic; for instance, see~\cite{Stallings}. Hence, $E(g)$ is the maximal cyclic subgroup containing $g$, that is, $E(g) = \langle a \rangle$ for some $a \in G$. Clearly, $\langle a \rangle\subseteq C_{G}(g) $. On the other hand, it follows from Theorem~\ref{thm:Osin} (2) that $C_{G}(g) \subseteq E(g)$; hence, $C_{G}(g) \subseteq \left\langle a\right\rangle $. So, $C_{G}(g) =\left\langle a\right\rangle $. 
\end{proof}

\begin{lemma}\label{lemma:abelian-subgp} Let $A$ be an abelian subgroup of $G$. Suppose that $x\in A$ has infinite order. If $x$ is parabolic, then $A$ is contained in a maximal parabolic subgroup; otherwise, $A$ is contained in $E(x)$.
\end{lemma}

\begin{proof}
If $x$ is parabolic, then $x \in P \cap A$ for some maximal parabolic subgroup $P$; Corollary \ref{cor:parabolic-element} implies that $A \subseteq P$ as $A$ is abelian. If $x$ is hyperbolic, then the centralizer $C_{G}(x)$ of $x$ is contained in $E(x)$ by Theorem \ref{thm:Osin} (2); and note that $A \subseteq C_{G}(x)$. Hence, $A \subseteq E(x)$.
\end{proof}

\begin{remark}
It is well-known that a hyperbolic group contains no free abelian subgroups of rank 2, i.e., $\mathbb{Z} \times \mathbb{Z}$. Torsion-free relatively hyperbolic groups can contain subgroups isomorphic to $\mathbb{Z} \times \mathbb{Z}$ but each one of these has to lie in some parabolic subgroup, by Lemma \ref{lemma:abelian-subgp} and Lemma \ref{lemma:elementary-hyp}.
\end{remark}

\begin{corollary}\label{cor:abelian-normal-subsubgp}
Let $G$ be torsion-free, let $H$ be a non-elementary non-parabolic subgroup of $G$, and let $A$ be an abelian subgroup of $H$ which is normal in $H$. Then $A$ is trivial. In particular, every non-elementary non-parabolic subgroup of a torsion-free relatively hyperbolic group has trivial centre.
\end{corollary}

\begin{proof}
Since $A$ is abelian, we have two cases by Lemma \ref{lemma:abelian-subgp}. First, $A$ is contained in some maximal parabolic subgroup $P$; and so, $A \subseteq P \cap H$. Since $H \not\subseteq P$, there exists $h \in H\setminus P$ such that $A \subseteq hPh^{-1}$ because $A$ is normal in $H$. Hence, $A \subseteq P \cap hPh^{-1}$, which implies that $A$ is trivial by Corollary \ref{cor:parabolic-malnormal}. The second case is when $A \subseteq E(x) \cap H$ for some hyperbolic element $x \in G$. Note that $H \not\subseteq E(x)$ and that $E(x)$ is almost malnormal by Theorem \ref{thm:Osin} (2). A similar argument as in the first case shows that $A$ is trivial. The last assertion follows immediately.
\end{proof}


\section{Residual homomorphisms}\label{sec:equivalent-def}

In this section we prove that Definitions \ref{def:resCclassical} and \ref{def:resCworking} over the class $\mathcal{R}$ are equivalent. Our proof is built upon results of Ol'shanskii and Osin. In \cite{Osin2010}, Osin generalizes the small cancellation theory over hyperbolic groups, developed by Ol'shanskii in \cite{Olshanskii93}, to relatively hyperbolic groups.

Let $G$ be a relatively hyperbolic group. Given a subgroup $H$ of $G$, we denote by $H^{0}$ the set of all hyperbolic elements of infinite order in $H$. Recall that two elements $f$ and $g$ in $G^{0}$ are said to be \emph{commensurable} in $G$ if $f^{k}$ is conjugate to $g^{l}$ in $G$, for some nonzero integers $k$ and $l$. 

\begin{definition}\rm
Let $G$ be a relatively hyperbolic group. A subgroup $H$ of $G$ is called \emph{suitable}, if there exist two non-commensurable elements $f_{1}$ and $f_{2}$ in $H^{0}$ such that $E(f_{1}) \cap E(f_{2}) = \{1\}$.
\end{definition}

\begin{proposition}\label{prop:AMOpr} \cite[Proposition 3.4]{AMO2007} Suppose that a group $G$ is hyperbolic relative to a collection of subgroups $\{H_{\lambda};\lambda \in \Lambda\}$ and that $H$ is a subgroup of $G$. Then $H$ is suitable if and only if $H^{0} \neq \emptyset$ and $E(H)=\bigcap_{g \in H^{0}}E(g) = \{1\}$.
\end{proposition}
\begin{lemma} \label{lemma:AMOl} \cite[Lemma 3.3]{AMO2007} With the notation of Proposition~\ref{prop:AMOpr}, if $H$ is a non-elementary subgroup of $G$ and $H_0 \neq\emptyset$, then $E(H)$ is the maximal
finite subgroup of $G$ normalized by $H$.
\end{lemma}
As an immediate consequence of Proposition~\ref{prop:AMOpr} and Lemma~\ref{lemma:AMOl} we have the following statement.
\begin{corollary}\label{cor:suitable}
Let a group $G$ be torsion-free hyperbolic relative to a collection of subgroups $\mathcal{H} = \{H_{\nu}, \nu \in \Lambda\}$. Suppose that $H$ is a non-elementary subgroup of $G$ containing a hyperbolic element $h$. Then $H$ is suitable. 
\end{corollary}

The following theorem is a combination of Theorem 2.4 and Lemma 5.1 from~\cite{Osin2010}.

\begin{theorem}\label{lemma:residual} (Osin)
Let $G$ be a group hyperbolic relative to a collection of subgroups $\{H_{\nu}, \nu \in \Lambda\}$, $X$ be a relative generating set for $G$ with respect to $\{H_{\nu}, \nu \in \Lambda\}$ and $\mathscr{H} = \bigsqcup_{\nu \in \Lambda}(H_{\nu} \setminus \{1\})$.
Let $H$ be a suitable subgroup of $G$, $t_{1}, \cdots, t_{m}$ arbitrary elements of $G$ and $N$ an arbitrary positive integer. Then there exists an epimorphism $\eta : G \rightarrow \bar{G}$ such that :
\begin{itemize}
\item[(1)] The group $\bar{G}$ is hyperbolic relative to $\{\eta(H_{\nu}), \nu \in \Lambda\}$.
\item[(2)] For any $i = 1, \cdots , m$, we have $\eta(t_{i}) \in \eta(H)$.
\item[(3)] The restriction of $\eta$ to $\mathscr{H}$ is injective.
\item[(4)] The restriction of $\eta$ to the subset of elements of length at most $N$ with respect to $X\cup \mathscr{H}$ is injective.
\item[(5)] If all hyperbolic elements of $G$ have infinite order then all hyperbolic elements of $\bar{G}$ have infinite order.
\end{itemize}
\end{theorem}

Recall that we denote by $\mathcal{R}$ the class of finitely generated toral relatively hyperbolic groups (see Definition~\ref{def:toral}). Note that every free abelian group of finite rank is hyperbolic relative to itself. Hence, the following lemma holds.

\begin{lemma}\label{lemma:free-abelian-res}
Every free abelian group of finite rank is fully residually-$\mathcal{R}$ in the sense of both Definition \ref{def:resCclassical} and \ref{def:resCworking}.
\end{lemma}

We now are able to prove the main result of this section.

\begin{theorem}\label{thm:res-subgp}
Let $G \in \mathcal{R}$. Every nontrivial finitely generated subgroup $H$ of $G$ is fully residually-$\mathcal{R}$ in the sense of Definition \ref{def:resCclassical}.
\end{theorem}
\begin{proof}  
A torsion-free elementary group is necessarily cyclic and so belongs to $\mathcal{R}$. So, let $H$ be non-elementary. If $H$ is contained in a parabolic subgroup of $G$, then $H$ is a free abelian group of finite rank, hence it is fully residually-$\mathcal{R}$ in the sense of Definition~\ref{def:resCclassical} by Lemma \ref{lemma:free-abelian-res}. If $H$ contains a hyperbolic element then $H$ is suitable by Corollary~\ref{cor:suitable}. Since $G$ is finitely generated, Theorem~\ref{lemma:residual} (2) implies that there exists a homomorphism $\eta : G \rightarrow \bar{G}$ such that $\eta(H)=\bar{G}$. Indeed, one can choose the elements $t_1,\dots,t_m$ to be a generating set for $G$. Furthermore,  Theorem~\ref{lemma:residual} (4) yields that $\eta$ is injective on any finite set of elements in $G$. By Theorem~\ref{lemma:residual} (1), (3) and (5), $\bar{G}\in \mathcal{R}$. Hence, by considering the homomorphism $\eta$ restricted to $H$, $H$ is fully residually-$\mathcal{R}$ in the sense of Definition~\ref{def:resCclassical}.
\end{proof}

\begin{corollary}
A finitely generated group $G$ is fully residually-$\mathcal{R}$ in the sense of Definition \ref{def:resCclassical} if and only if it is fully residually-$\mathcal{R}$ in the sense of Definition \ref{def:resCworking}.
\end{corollary}

\begin{proof}
Obviously, Definition \ref{def:resCclassical} implies Definition \ref{def:resCworking}. To prove the converse, let $L$ be a finitely generated group, and let $a_1,\dots,a_n$ be nontrivial elements of $L$. If $L$ is fully residually-$\mathcal{R}$ in the sense of Definition \ref{def:resCworking} then there is a group $G\in\mathcal{R}$ and a homomorphism $\phi: L\rightarrow G$ such that $\phi(a_i)\neq 1$, for all $i=1,\dots,n$. We argue that there is an epimorphism $\alpha$ from $L$ onto a group in $\mathcal{R}$ such that the elements $a_1,\dots,a_n$ survive in the image, as follows. Let $H=\phi(L)\subseteq G$. If $H$ is either free abelian, or elementary, or $H=G$ then $H\in\mathcal{R}$ and we are done. Otherwise, $H$ is a proper subgroup of $G$ and is suitable. Choose a finite generating set $t_1,\dots,t_m$ for $G$ and an integer $N$ such that  $N>\max\{|\phi(a_i)|_G; i=1,\dots,n \}$, and let $\eta$ be a homomorphism from  Theorem~\ref{thm:res-subgp} such that $\eta(H)=\bar{G}\in\mathcal{R}$. Let $\mu=\eta|_H$ be the restriction of $\eta$ to $H$. Then $\alpha=\mu\circ\phi : L\rightarrow \bar{G}$ is an epimorphism which has all of the required properties, by the proof of Theorem~\ref{thm:res-subgp}. Hence, $L$ is fully residually-$\mathcal{R}$ in the sense of Definition~ \ref{def:resCclassical}.
\end{proof}


\section{The generalization of Baumslag's theorem}\label{sec:BaumslagThm}
Relatively hyperbolic groups are not commutative transitive in general. For instance, if $G$ is a non-abelian relatively hyperbolic group with torsion, then it may have a nontrivial finite centre (cf. Corollary \ref{cor:abelian-normal-subsubgp}). In this case, $G$ is not commutative transitive. However, toral relatively hyperbolic groups are commutative transitive as the following theorem shows.

Recall that we denote by $\mathcal{R}$ the class of finitely generated toral relatively hyperbolic groups (see Definition~\ref{def:toral}).

\begin{theorem}\label{thm:CT-R}
If $G$ is a toral relatively hyperbolic group, then $G$ is commutative transitive. In particular, if $G \in \mathcal{R}$ then $G$ is commutative transitive.
\end{theorem}

\begin{proof}
Suppose that there are nontrivial elements $x, y$ and $z \in G$ with $[x,y] = 1$ and $[y,z] = 1$. It follows from Corollary \ref{cor:parabolic-element} that if one of $x,y,z$ is parabolic then so are the other two. Moreover, in this case all the three elements belong to one and the same maximal parabolic subgroup. Since every parabolic subgroup is abelian, we have $[x,z] = 1$.
So, suppose that $x,y$ and $z$ are hyperbolic in $G$. Since $yxy^{-1}=x$, $y \in E(x)$ and so, $E(x) \subseteq E(y)$ by Theorem \ref{thm:Osin} (2). Similarly, we have $E(y) \subseteq E(x)$. So, $E(x) = E(y)$. By a similar argument, we have $E(y) = E(z)$. Therefore, $x, y$ and $z \in E(x)$. But Lemma \ref{lemma:elementary-hyp} implies that $E(x)$ is the maximal cyclic subgroup containing $x$; and hence, $[x, z] = 1$. Thus, $G$ is commutative transitive.
\end{proof}

Throughout the rest of this section, we use Definition \ref{def:resCworking} of (fully) residually-$\mathcal{R}$ groups. We have shown in Section~\ref{sec:equivalent-def} that if $L$ is a finitely generated group then $L$ is (fully) residually-$\mathcal{R}$ in the sense of Definition~\ref{def:resCclassical} if and only if $L$ is (fully) residually-$\mathcal{R}$ in the sense of Definition~\ref{def:resCworking}.

\begin{proposition}\label{prop-1}
Suppose that $G$ is a fully residually-$\mathcal{R}$ group. Then $G$ is commutative transitive.
\end{proposition}

\begin{proof}
Let $x, y$ and $z$ be nontrivial elements in $G$; and let $[x,y] = 1$ and $[y,z] = 1$. We need to show that $[x,z] = 1$. Suppose that $[x,z] \neq 1$. Since $G$ is fully residually $\mathcal{R}$, there is a homomorphism $\phi : G \rightarrow H \in \mathcal{R}$ such that $\phi(x), \phi(y), \phi(z)$ and $\phi([x,z])$ are nontrivial elements in $H$. But $\phi([x,y])$ and $\phi([y,z])$ are trivial in $H$. Since $H$ is commutative transitive by Theorem \ref{thm:CT-R}, we have $\phi([x,z]) = 1$. The contradiction shows that $[x,z] = 1$. 
\end{proof}

\begin{lemma}\label{lemma:res-torsionfree}
If $G$ is a residually-$\mathcal{R}$ group, then $G$ is torsion-free.
\end{lemma}

\begin{proof}
Assume that there exists a nontrivial torsion element $x \in G$. Since $G$ is residually-$\mathcal{R}$, there is a homomorphism $\phi : G \rightarrow H \in \mathcal{R}$ such that $\phi(x) \neq 1$. Then $\phi(x)$ has infinite order in $H$, a contradiction. Therefore, $G$ is torsion-free.
\end{proof}

\begin{corollary}\label{abelian-fullyres}
Suppose that a finitely generated abelian group $G$ is residually-$\mathcal{R}$. Then $G$ is fully residually-$\mathcal{R}$.
\end{corollary}

\begin{proof}
By Lemma \ref{lemma:res-torsionfree}, $G$ has no torsion elements hence, it is isomorphic to a free abelian group. Then by Lemma~\ref{lemma:free-abelian-res}, $G$ is fully residually-$\mathcal{R}$.
\end{proof}

\begin{lemma}\label{lemma:abelian-normal-centre}
Every abelian normal subgroup of a residually-$\mathcal{R}$ group is contained in its centre.
\end{lemma}

\begin{proof}
Let $G$ be a residually-$\mathcal{R}$ group; let $A$ be an abelian normal subgroup of $G$; and let $Z(G)$ be the centre of $G$. Suppose that $A$ is not a subset of $Z(G)$. Then there exists a nontrivial element $a \in A \setminus Z(G)$ with $[a,g] \neq 1$ for some $g \in G$. Note that $[a,g] \in A$ as $A$ is normal in $G$. Since $G$ is residually-$\mathcal{R}$, there exists a homomorphism $\phi : G \rightarrow H \in \mathcal{R}$ such that $\phi([a,g]) \neq 1$; necessarily, $\phi(a)\neq 1$ and $\phi(g)\neq 1$. Then $\phi(G)$ is a non-abelian subgroup of $H$; in particular, $\phi(G)$ is a non-elementary non-parabolic subgroup of $H$. Note that $\phi(A)$ is abelian and normal in $\phi(G)$ as $A$ is abelian and normal in $G$. So, $\phi(A)$ is trivial by Corollary \ref{cor:abelian-normal-subsubgp}. It is impossible as $1 \neq \phi([a,g]) \in \phi(A)$. Therefore, we have $A \subseteq Z(G)$.
\end{proof}

\begin{example}\label{ex:F2xZ}
The direct product $F_{2} \times \mathbb{Z}$ of a free group of rank 2 and an infinite cyclic group is residually-$\mathcal{R}$. Note that $F_{2} \times \mathbb{Z}$ is not commutative transitive since it is non-abelian with a nontrivial centre. By Proposition \ref{prop-1}, $F_{2} \times \mathbb{Z}$ is not fully residually-$\mathcal{R}$.
\end{example}

Having proved that groups in $\mathcal{R}$ are commutative transitive, we can apply an argument due to B.~Baumslag~\cite{Baumslag65} to prove the following.

\begin{proposition}\label{prop-2}
Suppose that a finitely generated group $G$ is residually-$\mathcal{R}$ and commutative transitive. Then $G$ is fully residually-$\mathcal{R}$.
\end{proposition}

\begin{proof} (B. Baumslag)
Assume that the centre of $G$ is trivial; otherwise, $G$ would be abelian because it is commutative transitive and hence, $G$ is fully residually-$\mathcal{R}$ by Corollary \ref{abelian-fullyres}.

Let $n$ be a positive integer and let $g_{1}, \dots g_{n}$ be nontrivial elements of $G$. We claim that there exists a nontrivial $g \in G$ such that if $\phi(g) \neq 1$ for a homomorphism $\phi : G \rightarrow H \in \mathcal{R}$ then none of the elements $\phi(g_{1}), \cdots \phi(g_{n})$ is trivial. We can prove this by induction on $n$. For $n = 1$, it is true by taking $g = g_{1}$. Now, assume that there  exists a nontrivial $g \in G$ such that the claim is true for nontrivial elements $g_{1}, \dots, g_{k}$ where $k < n$. We need to show that there exists a nontrivial $g' \in G$ such that the assertion holds for $g_{1}, \dots , g_{k+1}$. Consider $c(x) = [g, xg_{k+1}x^{-1}]$ for $x \in G$. If $c(x) = 1$ for all $x$, then $A = \langle xg_{k+1}x^{-1}; \text{ for every } x \in G \rangle$ is abelian because $G$ is commutative transitive; in addition, $A$ is normal in $G$. Since the centre $Z(G)$ of $G$ is trivial, and by Lemma \ref{lemma:abelian-normal-centre} $A\subset Z(G)$, $A$ is trivial. But note that $g_{k+1} \in A$, which is impossible as $g_{k+1} \neq 1$. Hence, there exists $x$ in $G$ such that $c(x) \neq 1$. Choose $g' = c(x)$. Note that $\phi(g') \neq 1$ implies that $\phi(g) \neq 1$ and $\phi(g_{k+1}) \neq 1$. So, none of the elements $\phi(g_{1}), \cdots , \phi(g_{k+1})$ is trivial; and hence, the claim is proven.

Let $g$ be constructed as in the claim. Since $G$ is residually-$\mathcal{R}$, we have a homomorphism $\phi : G \rightarrow H \in \mathcal{R}$ such that $\phi(g) \neq 1$. We can conclude that $\phi(g_{i}) \neq 1$ for all $i = 1, \cdots , n$; in particular, $G$ is fully residually-$\mathcal{R}$.
\end{proof}

As a consequence of Propositions \ref{prop-1} and \ref{prop-2}, we have the following generalization of B. Baumslag's theorem~\cite{Baumslag65}.
\begin{theorem}
A finitely generated group is fully residually-$\mathcal{R}$ if and only if it is residually-$\mathcal{R}$ and commutative transitive. 
\end{theorem}


\section{Toral-limit groups}\label{sec:Toral-LimitGps}

In this section, we study the class of fully residually-$\mathcal{R}$ groups. We recall some definitions introduced by Miasnikov and Remeslennikov in the paper~\cite{MyasnikovRemeslennikov96}. Let $G$ be a group. The centralizer of an element $u \in G$ in $G$ is denoted by $C_{G}(u)$. 

\begin{definition} \label{def:ext-centralizers} \rm Let $G$ be a group. We say that a group $G_1$\textit{ is obtained from} $G$ \textit{by an extension of a centralizer} if, for some $u\in G$, $G_{1}$ is isomorphic to the free product of $G$ and $C_{G}(u) \times \langle t \rangle$ with amalgamation: $$G_1 \cong G *_{C_{G}(u)}(C_{G}(u) \times \langle t \rangle).$$ Also, $G_{1}$ has a presentation as an HNN-extension, as follows:  $$G_1=\langle G, t\mid [C_G(u),t] = 1\rangle.$$ A group obtained as the union of a chain of extensions of centralizers $$G = G_{0} \leq G_{1} \leq \cdots \leq G_{i} \leq G_{i+1} \leq \cdots \leq G_{k}$$ is called an \emph{iterated extension of centralizers up to} $k$, where $$G_{i+1} = \langle G_{i}, t_{i}\mid [C_{G_{i}}(u_{i}),t_{i}] = 1\rangle, \ i=0,1,\dots,k-1.$$
\end{definition}

The following result is proved by Osin in \cite{Osin06}.

\begin{lemma}\label{lemma:Osin-amalgated}
\cite[Corollary 1.5]{Osin06} Suppose that $A$ and $B$ are groups, hyperbolic relative to $\{A_{\mu}\}_{\mu \in M} \cup \{K\}$ and $\{B_{\nu}\}_{\nu \in N}$, respectively. Assume in addition that $K$ is finitely generated and for some $\eta \in N$, there is a monomorphism $\xi : K \rightarrow B_{\eta}$. Then the amalgamated product $A *_{K=\xi(K)} B$ is hyperbolic relative to $\{A_{\mu}\}_{\mu \in M} \cup \{B_{\nu}\}_{\nu \in N}$.
\end{lemma}

\begin{corollary}\label{cor:amal-hyp-rel}
Let $G$ be a finitely generated group hyperbolic relative to a collection $\mathcal{H} = \{H_{\lambda}, \lambda \in \Lambda\}$ of subgroups; and let $\mathcal{P} = \mathcal{H} \setminus\{ H_{\lambda}\}$ for some $\lambda \in \Lambda$. Suppose that $C$ is a subgroup of $H_{\lambda}$, and let $K = \langle H_{\lambda}, t \mid [C,t] = 1\rangle$. Then $G *_{C} (C \times \langle t \rangle)$ is hyperbolic relative to $\mathcal{P}\cup \{K\}$.
\end{corollary}

\begin{proof}
Since $G$ is finitely generated, each subgroup in $\mathcal{H}$ is finitely generated, by~\cite[Theorem~1.1]{OsinBook}. Note that $K$ is hyperbolic relative to itself and that $H_{\lambda}$ embeds into $K$. Hence, the amalgamated free product $G_1 = G *_{H_{\lambda}} K$ is hyperbolic relative to $\mathcal{P}\cup \{K\}$, by Lemma~\ref{lemma:Osin-amalgated}. Clearly, $G_1$ is isomorphic to the amalgamated product $G *_{C} (C \times \langle t \rangle)$. Therefore, $G *_{C} (C \times \langle t \rangle)$ is hyperbolic relative to $\mathcal{P}\cup \{K\}$.
\end{proof}

\begin{proposition}\label{extension-centralizers}  
Let $\mathcal{C}$ be the class of finitely generated torsion-free relatively hyperbolic groups, and let $G \in \mathcal{C}$. Suppose that $G_{k}$ is obtained from $G$ by an iterated extension of centralizers up to $k$. Then $G_{k} \in \mathcal{C}$, for any $k\geq 1$. In fact, if $G \in \mathcal{R}$ then $G_{k} \in \mathcal{R}$, for any $k\geq 1$.
\end{proposition}

\begin{proof}
Let $G \in \mathcal{C}$ be hyperbolic relative to a collection $\mathcal{H} = \{H_{\lambda}, \lambda \in \Lambda\}$ of subgroups, and let $u$ be a hyperbolic element in $G$. As $G$ is torsion-free, by Lemma~\ref{lemma:elementary-hyp}, $E(u) = \langle a \rangle$ for some $a\in G$, and so $C_{G}(u) = \langle a \rangle$. By Theorem~\ref{thm:Osin} (3), $G$ is hyperbolic relative to $\mathcal{H} \cup \{\langle a \rangle\}$. Corollary \ref{cor:amal-hyp-rel} implies that $$G_{1}=\left\langle G,t\mid [a,t]=1\right\rangle =G\ast_{\left\langle a\right\rangle} (\left\langle a\right\rangle \times \left\langle t\right\rangle )$$ is hyperbolic relative to $\mathcal{H} \cup (\langle a \rangle \times \langle t \rangle)$. As an amalgamated product of finitely generated torsion-free groups, $G_{1}$ is finitely generated and torsion-free. It follows that $G_{1} \in \mathcal{C}$. Note that if $G \in \mathcal{R}$ then the subgroups in $\mathcal{H}$ and $\langle a \rangle \times \langle t \rangle$ are all finitely generated free abelian and so $G_{1} \in \mathcal{R}$. 

Now, suppose that $u$ is a parabolic element in $G$, then $u' = g^{-1}ug \in H_{\lambda}$ for some $H_{\lambda} \in \mathcal{H}$ and $g \in G$. Let $\mathcal{P} = \mathcal{H} \setminus \{H_{\lambda}\}$. Denote $C=C_{G}(u')$; by Corollary~\ref{cor:parabolic-element}, $C \subseteq H_{\lambda}$. Set $K = \langle H_{\lambda}, t; [C,t] = 1\rangle$, then $K$ is finitely generated and torsion free. It follows that $G' = G\ast_{H_{\lambda}} K$ is finitely generated and torsion-free; and so is $G_{1} = G *_{C} (C \times \langle t \rangle)$ as $G_{1} \cong G'$. By Corollary~\ref{cor:amal-hyp-rel}, $G_{1}$ is hyperbolic relative to $\mathcal{P}\cup \{ K\}$.  Thus, $G_{1} \in \mathcal{C}$. Note that if $G \in \mathcal{R}$ then $C = H_{\lambda}$ and so $K$ is abelian; therefore, $G_{1} \in \mathcal{R}$. 

An easy induction on $k$ shows that $G_{k} \in \mathcal{C}$. If $G \in \mathcal{R}$ then either $C_{G_{i}}(u_{i})$ is a maximal cyclic hyperbolic subgroup, or $C_{G_{i}}(u_{i}')$ is a maximal abelian parabolic subgroup, for each $i$. By induction, we conclude that $G_{k} \in \mathcal{R}$, for all $k$.
\end{proof}

Observe that every nontrivial parabolic subgroup of $G \in \mathcal{R}$ is proper if and only if $G$ is non-abelian. We denote the class of non-abelian groups in $\mathcal{R}$ by $\mathcal{G}$. The statement of Theorem C from \cite{KharlampovichMyasnikov09} is equivalent to the following (see \cite{BaumsalgMyasnikovRemeslennikov99, MyasnikovRemeslennikov96}).

\begin{theorem}\label{thm:KM} \cite[Theorem C]{KharlampovichMyasnikov09}  
Suppose that $\Gamma \in \mathcal{G}$. Then a finitely generated group $G$ is fully residually-$\Gamma$ if and only if $G$ embeds into a group $\bar{\Gamma}$ obtained from $\Gamma$ by an iterated extension of centralizers up to a finite number $k$.
\end{theorem}

\begin{theorem}\label{thm:embedding-R}
Let $\Gamma \in \mathcal{R}$. If a finitely generated group $G$ is fully residually-$\Gamma$ then $G$ embeds into a group $\bar{\Gamma}\in\mathcal{R}$. 
\end{theorem}

\begin{proof}
Note that $\mathcal{R} = \mathcal{G} \cup \{\text{all abelian groups in } \mathcal{R}\}$. If $\Gamma \in \mathcal{G}$, then the statement immediately follows from Proposition~\ref{extension-centralizers} and Theorem~\ref{thm:KM}. If $\Gamma\in\mathcal{R}$ is a finitely generated free abelian group then $G$ is free abelian and so $G \in \mathcal{R}$.
\end{proof}


\section{The Baumslag-Solitar groups}\label{sec:BS-gps}

The statement of Theorem \ref{thm:embedding-R} can be contrasted with the case when one considers torsion-free fully residually-$\mathcal{H}$ groups, where $\mathcal{H}$ is the class of hyperbolic groups, see Proposition~\ref{prop:BS-fully-res-H} below.

Recall that the presentation $BS(p,q)=\left<a,b; ba^{p}b^{-1} = a^{q}\right>$ with $p,q\in\mathbb{Z}$ defines a Baumslag-Solitar group. Whereas the integers $p$ and $q$ can be arbitrary, we always assume that $p$ and $q$ are nonzero.

The following lemma is stated as Corollary 4.22 in \cite{OsinBook}. We provide a proof for completeness.

\begin{lemma}\label{lemma:Osin-BS} \cite[Corollary 4.22]{OsinBook} 
Let $G$ be a finitely generated group hyperbolic relative to a collection of subgroups $\{H_{1}, \cdots , H_{m}\}$. Suppose that $$B \cong BS(p,q)=\left<a,b; ba^{p}b^{-1} = a^{q}\right>$$ with $pq\neq 0$ is a subgroup of $G$. Then $B$ is a parabolic subgroup of $G$.
\end{lemma}

\begin{proof} 
The assumption $pq\neq 0$ implies that $B$ is torsion-free.

Firstly, assume that $a$ is parabolic. Then $a$ is contained in a maximal parabolic subgroup $P$, and we have that $a^{q} \in P \cap bPb^{-1}$ contradicting to Corollary \ref{cor:parabolic-malnormal}, unless $b \in P$. So, $B$ is a parabolic subgroup. 

Now, assume that $a$ is hyperbolic. Then $a \in E(a)$ by Theorem \ref{thm:Osin} (2), and $E(a)$ is hyperbolic as an abstract group. By~\cite[Corollary 4.21]{OsinBook}, $q = \pm p$ and so, by Theorem~\ref{thm:Osin} (2), $b \in E(a)$. Therefore, $B\cong \left<a,b; ba^{p}b^{-1} = a^{\pm p}\right>$ is a subgroup of $E(a)$. However, the subgroup $\left\langle a^{p},b^2\right\rangle $ of $B$ is isomorphic to a free abelian group of rank 2, which is a contradiction. Therefore, $a$ cannot be a hyperbolic element of $G$.
\end{proof}

\begin{corollary}\label{cor:BS}
Let $B \cong BS(1,q)$, where $ q \geq 2$. Then $B$ is not a subgroup of a group in $\mathcal{R}$.
\end{corollary}

\begin{proof}
Let $G$ be a group in $\mathcal{R}$. Assume that $B$ is a subgroup of $G$. Then $B$ is a parabolic subgroup by Lemma~\ref{lemma:Osin-BS}; hence, $B$ is abelian, which is a contradiction.
\end{proof}

The following Theorem is due to Meskin~\cite{Meskin}.

\begin{theorem}\label{lemma:Meskin} \cite[Theorem C]{Meskin} 
$BS(p,q)$ is residually finite if and only if $|p|=1$ or $|q|=1$ or $|p|=|q|$.
\end{theorem}

\begin{proposition}\label{prop:BS-fully-res-H}
Let $\mathcal{H}$ denote the class of hyperbolic groups. Let $B = BS(1,q)$ where $q \geq 2$. Then $B$ is torsion-free fully residually-$\mathcal{H}$. In particular, torsion-free fully residually-$\mathcal{H}$ groups do not embed into groups in $\mathcal{R}$.
\end{proposition}

\begin{proof}
Note that $B$ is torsion-free. By Theorem~\ref{lemma:Meskin}, $B$ is residually finite. Hence, $B$ is fully residually finite. It follows that $B$ is fully residually-$\mathcal{H}$, as the class $\mathcal{H}$ contains all finite groups. Note that $B$ cannot be embedded into a group from $\mathcal{R}$ by Corollary~\ref{cor:BS}. Therefore, the result follows.
\end{proof}

Note that, whereas the groups $B_q = BS(1,q)$, $q\geq 2$, are torsion-free, in our proof the hyperbolic quotients of $B_q$ all have torsion elements. We ask the following question, which can be contrasted with the statement of Theorem~\ref{thm:embedding-R}.

\begin{question}
Let $G$ be a finitely generated fully residually-$\mathcal{R}$ group. Does $G$ embed into a group in $\mathcal{R}$?
\end{question}

\bibliographystyle{plain}
\bibliography{OnFullyResRGps}

\begin{thebibliography}{10}

\bibitem{Alibegovic}
E.~Alibegovic.
\newblock A combination theorem for relatively hyperbolic groups.
\newblock {\em Bulletin of the London Mathematical Society}, 37(3):459--466,
  2005.

\bibitem{AMO2007}
G.~Arzhantseva, A.~Minasyan, and D.~Osin.
\newblock The \text{SQ}-universality and residual properties of relatively
  hyperbolic groups.
\newblock {\em Journal of Algebra}, 315(1):165--177, 2007.

\bibitem{Baumslag65}
B.~Baumslag.
\newblock Residually free groups.
\newblock {\em Prceedings of the London Mathematical Society}, 17(3):402--418,
  1967.

\bibitem{BaumsalgMyasnikovRemeslennikov99}
G.~Baumslag, A.~Maysnikov, and V.~Remeslennikov.
\newblock Algebraic geometry over groups \text{I}: Algebraic sets and ideal
  theory.
\newblock {\em Journal of Algebra}, 219:16--79, 1999.

\bibitem{Bowditch98}
B.~H. Bowditch.
\newblock Relatively hyperbolic groups.
\newblock {\em International Journal of Algebra Computation}, 22, 2012.

\bibitem{CFR}
L.~Ciobanu, B.~Fine, and G.~Rosenberger.
\newblock Classes of groups generalizing a theorem of \text{B}enjamin
  \text{B}aumslag.
\newblock {\em Communications in Algebra}, 44(2):656--667, 2016.

\bibitem{DahmaniCombinationThm}
F.~Dahmani.
\newblock Combination of convergence groups.
\newblock {\em Geometry \text{\&} \text{T}opology}, 7:933--963, 2003.

\bibitem{DahmaniGuirardel}
F.~Dahmani and V.~Guirardel.
\newblock Presenting parabolic subgroups.
\newblock {\em Algebraic \text{\&} \text{G}eometric \text{T}opology},
  13(6):3203--3222, 2013.

\bibitem{Farb98}
B.~Farb.
\newblock Relatively hyperbolic groups.
\newblock {\em Geometric And Functional Analysis}, 8(5):810--840, 1998.

\bibitem{GrovesLimitRelHypGps05}
D.~Groves.
\newblock Limit groups for relatively hyperbolic groups. {II}.
  {M}akanin-{R}azborov diagrams.
\newblock {\em Geom. Topol.}, 9:2319--2358, 2005.

\bibitem{KharlampovichMyasnikov98p1}
O.~Kharlampovich and A.~Myasnikov.
\newblock Irreducible affine varieties over a free group. {I}. {I}rreducibility
  of quadratic equations and {N}ullstellensatz.
\newblock {\em J. Algebra}, 200(2):472--516, 1998.

\bibitem{KharlampovichMyasnikov98}
O.~Kharlampovich and A.~Myasnikov.
\newblock Irreducible affine varieties over a free group. {II}. {S}ystems in
  triangular quasi-quadratic form and description of residually free groups.
\newblock {\em J. Algebra}, 200(2):517--570, 1998.

\bibitem{KharlampovichMyasnikov06}
O.~Kharlampovich and A.~Myasnikov.
\newblock Elementary theory of free non-abelian groups.
\newblock {\em J. Algebra}, 302(2):451--552, 2006.

\bibitem{KharlampovichMyasnikov09}
O.~Kharlampovich and A.~Myasnikov.
\newblock Limits of relatively hyperbolic groups and \text{L}yndon's
  completions.
\newblock {\em Journal of the European Math. Soc.}, 14:659--680, 2012.

\bibitem{KharlampovichMyasnikov13}
O.~Kharlampovich and A.~Myasnikov.
\newblock Decidability of the elementary theory of a torsion-free hyperbolic
  group.
\newblock Preprint, 2013.

\bibitem{Meskin}
S.~Meskin.
\newblock Nonresidually finite one-relator groups.
\newblock {\em Transactions of the American Mathematical Society},
  164:105--114, 1972.

\bibitem{MyasnikovRemeslennikov96}
A.~Myasnikov and V.~Remeslennikov.
\newblock Exponential groups \text{II} : extensions of centralizers and tensor
  completion of \text{CSA}-groups.
\newblock {\em International Journal of Algebra Computation}, 6(6):687--711,
  1996.

\bibitem{Olshanskii93}
A.~Yu. Ol'shanskii.
\newblock On residualing homomorphisms and \text{$G$}-subgroups of hyperbolic
  groups.
\newblock {\em International Journal of Algebra Computation}, 3:365--409, 1993.

\bibitem{Osin04}
D.~V. Osin.
\newblock Elementary subgroups of relatively hyperbolic groups and bounded
  generation.
\newblock {\em International Journal of Algebra Computation}, 16(01):99--118,
  2006.

\bibitem{Osin06}
D.~V. Osin.
\newblock Relative \text{Dehn} functions of \text{HNN}-extensions and
  amalgamated products.
\newblock {\em Contemporary Mathematics}, 394:209--220, 2006.

\bibitem{OsinBook}
D.~V. Osin.
\newblock Relatively hyperbolic groups: Intrinsic geometry, algebraic
  properties, and algorithmic problems.
\newblock {\em Memoirs AMS}, 179(843), 2006.

\bibitem{Osin2010}
D.~V. Osin.
\newblock Small cancellations over relatively hyperbolic groups and embedding
  theorems.
\newblock {\em Annals of mathematics}, 172:1--39, 2010.

\bibitem{OuldH}
Abderezak Ould~Houcine.
\newblock Limit groups of equationally {N}oetherian groups.
\newblock In {\em Geometric group theory}, Trends Math., pages 103--119.
  Birkh\"auser, Basel, 2007.

\bibitem{RW}
C.~Reinfeldt and R.~Weidmann.
\newblock Makanin-\text{R}azborov diagrams for hyperbolic groups.
\newblock Preprint, 2010.

\bibitem{SelaMRDiagrams01}
Z.~Sela.
\newblock Diophantine geometry over groups. {I}. {M}akanin-{R}azborov diagrams.
\newblock {\em Publ. Math. Inst. Hautes \'Etudes Sci.}, 93:31--105, 2001.

\bibitem{SelaElemTheoryFreeGps06}
Z.~Sela.
\newblock Diophantine geometry over groups. {VI}. {T}he elementary theory of a
  free group.
\newblock {\em Geom. Funct. Anal.}, 16(3):707--730, 2006.

\bibitem{SelaElemTheoryHypGps}
Z.~Sela.
\newblock Diophantine geometry over groups. {VII}. {T}he elementary theory of a
  hyperbolic group.
\newblock {\em Proc. Lond. Math. Soc. (3)}, 99(1):217--273, 2009.

\bibitem{Stallings}
John~R. Stallings.
\newblock On torsion-free groups with infinitely many ends.
\newblock {\em Ann. of Math. (2)}, 88:312--334, 1968.

\end{thebibliography}
\end{document}